\documentclass[11pt,thmsa]{article}%
\usepackage{amssymb,amsmath,latexsym}
\usepackage{amsfonts}
\usepackage{txfonts}
\usepackage{dsfont}
\usepackage{amssymb}
\usepackage{graphicx}%
\newtheorem{theorem}{Theorem}[section]

\newtheorem{definition}[theorem]{Definition}

\newtheorem{lemma}[theorem]{Lemma}

\newenvironment{proof}[1][Proof]{\noindent\textbf{#1.} }{\ \rule{0.5em}{0.5em}}

\begin{document}
\title{Clifford-Wolf homogeneous Randers spaces\footnote{Supported by NSFC (no.10971104) and SRFDP of China.}}
\author{  Shaoqiang Deng$^{1}$ and Ming Xu$^2$
\thanks{Corresponding author. E-mail: mgxu@math.tsinghua.edu.cn}\\
$^1$School of Mathematical Sciences and LPMC\\
Nankai University\\
Tianjin 300071, P. R. China\\
$^2$Department of Mathematical Sciences\\
Tsinghua University\\
Beijing 100084, P. R. China}
\date{}
\maketitle
\begin{abstract}
A Clifford-Wolf translation of a connected Finsler space  is an
isometry which moves each point the same distance. A Finsler space
$(M, F)$ is called Clifford-Wolf homogeneous if for any two points
$x_1, x_2\in M$ there is a Clifford-Wolf translation $\rho$ such
that $\rho(x_1)=x_2$. In this paper, we give a complete
classification of connected simply connected Clifford-Wolf homogeneous Randers
spaces.


\end{abstract}

\maketitle
\section{Introduction}
Let $(M, F)$ be a connected Finsler space and $d$ its distance
function. An isometry $\rho$ of $(M, F)$ is called  a Clifford-Wolf
translation (CW-translation for short) if the function $d(x, \rho
(x))$ is constant on $M$. The Finsler space $(M, F)$ is called
Clifford-Wolf homogeneous (CW-homogeneous) if for any two points
$x_1, x_2\in M$, there is a CW-translation $\rho$ such that
$\rho(x_1)=x_2$.

These definitions are  natural generalizations of  the related
notions in Riemannian geometry. CW-translations have been studied
extensively in Riemannian geometry, due to its relevance in the
classification of space forms. Let $(N, Q_1)$ be a connected
complete Riemannian manifold with constant sectional curvature $k$.
Then $(N, Q_1)$ is isometric to a quotient manifold $(M, Q)/\Gamma$,
where $(M, Q)$ is a connected simply connected complete Riemannian
manifold with constant sectional curvature $k$, and $\Gamma$ is a
discrete discontinuous subgroup of the full group of isometries of
$(M, Q)$ which acts freely on $M$. It is  well-known  that a
connected simply connected complete Riemannian manifold of constant
curvature is homogeneous. However, the quotient manifold $(M,
Q)/\Gamma$ is no longer homogeneous in general. J. A. Wolf proved in
\cite{WO60} that it is homogeneous if and only of $\Gamma$ consists
of CW-translations. Thus the classification of homogeneous space
form is reduced to the study of CW-subgroups of the full group of isometries.
This formulation was generalized to symmetric Riemannian spaces by
J. A. Wolf in \cite{WO62}; see \cite{FR63, OZ74} for other proofs.

Recently, connected simply connected Riemannian CW-homogeneous manifolds were
classified by Berestovskii and  Nikonorov in \cite{BN081, BN082,
BN09}. Their list consists of the euclidean spaces, the
odd-dimensional spheres with constant curvature, the connected
simply connected compact simple Lie groups with bi-invariant
Riemannian metrics,  and the direct products of the above Riemannian manifolds.
The purpose of this paper is to give a classification of connected simply
connected CW-homogeneous Randers spaces. Our main result is the
following
\begin{theorem}\label{main}
A connected  simply connected Randers space $(M,F)$ is
CW-homogeneous if and only if $M$ is a product of odd dimensional
spheres, connected simply connected compact simple Lie groups and a euclidian
space, and the Finsler metric $F$ has the  navigation data $(h,W)$ such that $(M,h)$ is
CW-homogeneous, i.e., $h$ is a Riemannian product of  Riemannian  metrics on spheres of constant
curvature, bi-invariant metrics on Lie groups and
a flat metric on the euclidian space with respect to the product
decomposition of $M$, and $W$ is a Killing vector field of constant
length with $||W||_h<1$.
\end{theorem}

This result presents many examples of non-reversible CW-homogeneous
Finsler spaces. It follows easily from the main results of Bao-Robles-Shen \cite{BRS04} that, on an odd dimensional sphere, a
CW-homogeneous Randers space must be of constant flag curvature.
However, the converse  statement is not true. In fact, the Randers
space $(S^{2n-1},F)$ of constant flag curvature is CW-homogeneous if
and only if in its navigation data $(h,W)$, $h$ is the standard Riemannian metric and $W$ is a Killing field
of constant length with $||W||_h<1$ (\cite{DM2}). When $W$ is not zero
and of constant length, it defines a complex structure on
$\mathbb{R}^{2n}$, and correspondingly the sphere can be denoted as
$U(2n)/U(2n-1)$. It is easily seen that $W$ is the Killing vector field generating the
center of $U(2n)$ (\cite{WHD}).


In view of the above result, it would be an interesting
 problem to classify CW-homo-geneous Finsler spaces in general.

In Section 2, we recall the preliminary knowledge about Finsler
geometry, CW-translation and CW-homogeneity. In Section 3, we
prove an interrelation theorem between CW-translations and Killing
vector fields of constant length on a homogeneous Finsler space. In Section 4, we prove the main
theorem.

\section{Preliminaries}

Let $M$ be a connected manifold. A Finsler metric is a continuous
function $F:TM\rightarrow \mathbb{R}^+$, which is smooth on the slit
tangent bundle $TM\backslash 0$. In a standard local coordinates
$(x^i,y^j)$ for $TM$, where $x=(x^i)$ is the local coordinates for
$M$, and $y=y^j\partial_{x^j}$ is the linear coordinates for $y\in
T_x(M)$, the Finsler metric $F$ is required to satisfy the following
properties:
\begin{description}
\item{(1)}\quad $F(x,y)>0$ for any $y\neq 0$.
\item{(2)}\quad $F(x,\lambda y)=\lambda F(x,y)$ for any $y\in T_x(M)$ and $\lambda >0$.
\item{(3)}\quad The Hessian matrix defined by $g_{ij}=\frac{1}{2}[F^2]_{y^i y^j}$ is positive definite.
\end{description}

A Randers metric on $M$ is a Finsler
   metrics of the form $F=\alpha+\beta$, where $\alpha$ is a Riemannian
   metric and $\beta$ is a smooth $1$-form on $M$ whose length with respect to $\alpha$
   is everywhere less than $1$.
   This kind of metric was introduced by G. Randers in 1941 (\cite{RA41})
   in the context of general relativity.

The above expression of a Randers metric is called the defining form in the literature.
Using standard local coordinates $(x^i,y^j)$ for $TM$, a Randers
metric can be presented as
\begin{align}
 F=\alpha+\beta=\sqrt{a_{ij}(x)y^iy^j}+b_i(x)y^i.
 \label{def}
 \end{align}

There is another method introduced by Shen \cite{SH02} to express a
Randers metric. The main idea is that the Randers metric $F$ can
also be uniquely written as
$$F(x,y)=\frac{\sqrt{h(y, W)^2+\lambda h(y, y)}}{\lambda}-\frac{h(y,
W)}{\lambda},$$ where $h$ is a Riemannian metric,  $W$ is a
vector field on $M$ with $h(W,W)<1$ and  $\lambda=1-h(W, W)$.  The
pair $(h, W)$ is called the navigation data of the Randers metric
$F$. The navigation data  is convenient when handling some problems concerning the
flag curvatures and Ricci scalar of a Randers space. For example,
using navigation data, Bao, Robles and Shen presented a very explicit description of Einstein-Randers metrics and Randers spaces
of constant flag curvature; see \cite{BR04,
BRS04}.

In a local coordinate system, the transformation laws between the defining form and
navigation data can be described as the following. If
$$F=\alpha+\beta=\sqrt{a_{ij}y^iy^j}+b_iy^i,$$
then the navigation data has the form
$$h_{ij}= (1-||\beta||^2)(a_{ij}-b_ib_j),\quad
W^i=-\frac{a^{ij}b_j}{1-||\beta||^2_{\alpha}}.$$
Conversely, the defining form can also be expressed by the navigation data through  the formula:
$$ a_{ij}=\frac{h_{ij}}{\lambda}+\frac{W_i}{\lambda}\frac{W_j}{\lambda},\quad b_i=\frac{-W_i}{\lambda},\eqno{(1.2)}$$
here $W_i=h_{ij}W^j$ and $\lambda=1-W^iW_i=1-h(W,W)$. All these formulas can be found in \cite{BR04}.

A Finsler metric $F$ defines the length of a tangent vector, and the
integration along a piece-wise smooth path defines the arc length of curves.
Taking the infimum of arc lengths for all
 piece-wise smooth paths, one defines a distance
 function $d(\cdot,\cdot)$ on $(M, F)$, which  satisfies all the conditions of a distance in the general sense except  the reversibility.
Based on  the distance function, we have generalized the
concept of a Clifford-Wolf translation (CW-translation) to Finsler
geometry.

\begin{definition}
Let $(M, F)$ be a connected Finsler space. An isometry $\rho$ of $(M, F)$ is called a Clifford-Wolf translation   if the function $d(x, \rho(x))$  is a constant on $M$.
\end{definition}

One can also define the notions of  Clifford-Wolf homogeneous (CW-homogenous)
spaces and of restrictively Clifford-Wolf homogenous (restrictively
CW-homogeneous) spaces.

\begin{definition}
A Finsler manifold $(M,F)$ is called Clifford-Wolf homogeneous  if
for any two points $x,x'\in M$, there is a CW-translation $\rho$
such that $\rho(x)=x'$. It is called restrictively Clifford-Wolf
homogeneous if  for any point $x\in M$ there is a neighborhood $V$ of
$x$, such that for any two points $x_1, x_2\in V$
there is
 a CW-translation $\rho$ such that $\rho (x_1)=x_2$.
\end{definition}

From the above definitions, it is obvious that a Finsler manifold
$(M,F)$ is restrictively CW-homogeneous if it is CW-homogeneous. Moreover, a connected restrictively CW-homogeneous Finsler space
$(M, F)$  must be a homogeneous Finsler space, i.e.,  its isometry group
$G=I(M,F)$ acts transitively on $M$.  In fact, if it is connected and
restrictively CW-homogeneous, then each $G$-orbit in $M$ is an open  as well as a closed subset.

\section{Killing vector fields of constant length}

 CW-translations of a connected Finsler manifold $(M,F)$ have a
natural interrelation with the Killing vector fields of constant
length, which has been proved by Berestovskii and Nikonorov
(\cite{BN09}) in the Riemannian case, and by the authors in Finsler geometry
(\cite{DM1}).
\begin{theorem}(\cite{DM1})\label{relation-1}
Let $(M,F)$ be a complete Finsler manifold with a positive
injectivity radius. If $X$ is a Killing vector field of constant
length, then the flow $\varphi_t$ generated by $X$ is a
CW-translation for all sufficiently small $t>0$.
\end{theorem}

\begin{theorem}(\cite{DM1})\label{original-interrelation}
Let $(M, F)$ be a  compact Finsler manifold. Then there is a
$\delta>0$, such that any CW-translation $\rho$ with
$d(x,\rho(x))<\delta$ is generated by a Killing vector field of
constant length.
\end{theorem}

 Theorem
\ref{original-interrelation} is not convenient when dealing with non-compact Finsler spaces.  We now prove the following theorem which will be useful in proving the main theorem of this paper.

\begin{theorem}\label{new-interrelation}
Let $(M,F)$ be a connected homogeneous Finsler space. Then there is
a $\delta>0$, such that any CW-translation $\rho$ with
$d(x,\rho(x))\leq\delta$ is generated by a Killing vector field of
constant length.
\end{theorem}

\begin{proof}
The connected isometry group $G=I_0(M,F)$ is Lie group (\cite{DH02})
which acts transitively on $M$. Let $\mathfrak{g}$ be its Lie
algebra. The exponential map $\exp:\mathfrak{g}\rightarrow G$ is a
diffeomorphism when restricted to a small round open disk
neighborhood $B$ of $0\in \mathfrak{g}$. When the positive number $\delta$
is small enough, all the CW-translations $\rho$ with
$d(x,\rho(x))\leq\delta$ is contained in $\exp(B)$. If we denote the
set of all CW-translations $\rho$ with $d(x,\rho(x))\leq \delta$ as
$\mathcal{C}_\delta$, then both $\mathcal{C}_\delta$ and
$\exp^{-1}(\mathcal{C}_\delta)\cap B$ are compact. As the
homogeneous space $(M,F)$ has a positive injectivity radius $r>0$,
we can choose $B$ to be small enough such that for any $X\in B$,
$F(X(x))<r$.

Now any CW-translation $\rho\in\mathcal{C}_\delta$ is generated
by a unique Killing vector field $X\in B$, such that
$\rho=\varphi_{1;X}=\exp X$, where $\varphi_{t;X}$ is the local one-parameter subgroup of
diffeomorphisms generated by $X$. We need to prove that $F(X)$ is a
constant function on $M$. Let $g_n$ be a sequence in $G$ such that
$F(X(g_n^{-1}x))$ converge to $\sup F(X)$. The supremum of $F(X)$ is
not required to be finite, but we can see from the later argument that
$\sup F(X)\leq r$. We have a sequence of CW-translations
$\rho_n=\mathrm{Ad}(g_n)\rho$ in $\mathcal{C}_{\delta}$. Each
$\rho_n$ is generated by the Killing vector field
$X_n={g_{n}}_{*}X$, i.e., $\rho_n=\varphi_{1;X_n}=\exp X_n$, and we have
$F(X(g_n^{-1}x))=F(X_n(x))$. By the connectedness of $G$ one easily sees that all the Killing vector fields $X_n$ is
contained $\exp^{-1}(\mathcal{C}_\delta)\cap B$. Choosing subsequence if necessary, one can assume
that the sequence $\{\rho_n\}$ converges to $\rho'\in\mathcal{C}_\delta$, with
$d(\cdot,\rho'(\cdot))\equiv d(\cdot,\rho(\cdot))$, that $\{X_n\}$ converges
to a Killing vector field $X'\in\exp^{-1}\mathcal{C}_\delta\cap B$
and that $\rho'=\varphi_{1;X'}=\exp X'$. At each point, we have
$$X'(\cdot)=\lim_{n\rightarrow\infty}X_n(\cdot).$$
 In particular, at
the point $x$, the function
$$F(X'(\cdot))=\lim_{n\rightarrow\infty}F(X_n(\cdot))$$
reaches its maximum. If $F(X'(x))=0$, then
$X=0$ and $\rho$ is the identity map. Otherwise $\nabla^X_X
X=-\frac{1}{2}\tilde{\nabla}^{(X)}F(X)^2$ vanishes along the flow
curve $\varphi_{t;X'}(x)$, $t\geq 0$. Then by Corollary 3.2 of \cite{DM1},  the flow curve
$\varphi_{t;X'}(x)$, $t\geq 0$ is a geodesic. Because
$F(X'(\varphi_{t;X'}(x)))=F(X'(x))\leq r$ for all $t\geq 0$, the
geodesic $\varphi_{t;X'}(x)$, $t\in [0,1]$ is minimizing. Thus
$$F(X'(x))=\sup(F(X))=d(x,\rho'(x))=d(x,\rho(x)).$$
A similar argument
for $\inf(F(X))$ (which must be positive when the CW-translation is
not the identity map) gives $\inf(F(X))=d(x,\rho(x))$. Therefore $X$ is a
Killing vector field of constant length.
\end{proof}

As a corollary, we can give a description of restrictive
CW-homogeneity by Killing vector fields of constant length, which
will be used in the proof of the main theorem.

\begin{theorem}\label{2}
Let $(M,F)$ be a connected homogeneous Finsler space. Then it is
restrictively CW-homogeneous if and only if the Killing fields of
constant length can exhaust all tangent directions, or equivalently,
any geodesic starting from any point is the flow curve of a Killing
field of constant length.
\end{theorem}
\begin{proof} If $(M,F)$ is homogeneous, it has a
positive injectivity radius $r>0$. At each point, there is a small
neighborhood, in which for any points $x$ and $x'$, $d(x,x')<r$.
Then the proof of Theorem \ref{relation-1} indicates the existence
of a CW-translation $\rho$ which maps $x$ to $x'$. So $(M,F)$ is
restrictive CW-homogeneous.

If $(M,F)$ is connected and restrictively homogeneous, then it is
homogenous. Then Theorem \ref{new-interrelation} indicates that all tangent
vectors are exhausted by Killing vector fields of constant length.
\end{proof}

In the Riemannian case, the above arguments give an alternative
proof for Theorem 7 in \cite{BN09}.

\section{Proof of Theorem \ref{main}}

We first prove that, if $(M,F)$ is a connected simply connected
restrictively CW-homogeneous Randers space, then $M$ has a product
form with a metric $F$ as described in the theorem.

Let $(h, W)$ be the navigation data of $F$ and $\varphi_{t;W}$ be
the flow generated by the vector field $W$. By Theorem \ref{2}, for
any $y\in TM_x$, there is a Killing vector field $X$ of constant
length for $F$, such that $X(x)=y$. Let $x'\in M$ and $t\geq 0$,
such that $\varphi_{t;W}(x)=x'$. Since $X$ is a Killing vector field
for $F$, we have $L_X W=[X,W]=0$. Moreover, $y=X(x)$ has the same
$F$-length as $y'=X(x')={\varphi_{t;W}}_*(y)$. It follows that all
$\varphi_{t;W}$'s are isometries and
 $W$ is a Killing vector field of $F$. For any Killing vector field
 $X$ of constant length $1$ of $F$, $X+W$ is a Killing vector field
 for $h$. It is easily seen that   $X+W$ has
 constant length $1$ with respect to $h$. Since $(M,F)$ is restrictively
 CW-homogeneous, the set of Killing vector fields $X$'s of constant length $1$ of $F$
exhaust all tangent directions. Because the length of $W$ with
respect to $F$ is less than $1$, the set of all the vector fields
$X+W$ exhaust all tangent directions too. So $(M,h)$ is
restrictively CW-homogeneous. By the classification theorem in
\cite{BN09}, $(M,h)$ is a Riemannian product of odd dimensional
spheres with constant curvature metrics, compact connected simply
connected Lie groups with bi-invariant metrics, and a flat Euclidian
space. The
 vector field $W$ is of constant length with respect to $F$, so
it is of  constant length with respect to $h$.

Next we  prove that, if $M$ is the product manifold with a Randers
metric $F$ as described in the theorem, then $(M, F)$ is
restrictively CW-homogeneous.

If $(M,F)$ is the product manifold as described in the theorem, then
it is homogeneous. Let $h^2=\sum_{i=1}^n h_i^2$ be the decomposition
for the symmetric metric, with respect to the decomposition
$M=M_1\times\cdots\times M_n$, and denote $W=\sum_{i=1}^n W_i$. Then
each $W_i$ is a Killing field of constant length for $h_i$. A
diffeomorphism $\psi=\psi_1\times\cdots\times\psi_n\in I(M,h)$ in
which each $\psi_i\in I(M_i,h_i)$ is an isometry for $(M,F)$ if and
only if $\psi_i$ keeps $W_i$ invariant for each $i$. We only need to
check that it acts transitively for each factor. If $W_i=0$, then the
argument for the corresponding factor is trivial. Thus we
assume $W_i\neq 0$. If $(M_i,h_i)$ is a compact connected simply
connected simple Lie group and $h_i$ is bi-invariant, then $W_i$
must belong to the Lie algebra for $L(M_i)$ or $R(M_i)$. Without
losing generality, we assume $W_i$ belongs to the Lie algebra for
$L(M_i)$, then all right translations acts transitively on $M_i$. If
$(M_i,h_i)$ is an odd dimensional sphere $S^{2k-1}$ with constant
curvature metric, then $W_i$ defines a complex structure $J$ such
that $W_i=cJ$ with $c\neq 0$, when both are regarded as matrices in
${\mathfrak so}(2k)$. Then the isometries keeping $W_i$ invariant are just the
group ${\mathrm U}(k)$ with respect to the complex structure $J$, and it is  obvious that
${\mathrm U}(k)$ act transitively on $M_i$. If $(M_i,h_i)$ is a flat euclidian space,
then $W_i$ is a constant vector field, and $\psi_i$ can be any parallel
translation, which also acts transitively on $M_i$.

We now prove that the set of  Killing fields of constant length for $h$
which commute with $W$ exhaust all tangent directions at each point. Any
Killing fields $X$ of constant length for $h$ can be decomposed as
 $X=\sum_{i=1}^n X_i$, in which each $X_i$ is a Killing
field of constant length for $h_i$. The condition $[X,W]=0$ is
equivalent to $[X_i,W_i]=0$ for each $i$. So the discussion breaks
down to a case by case discussion for each factor $(M_i,h_i)$. We
only consider the  case  that $W_i\neq 0$. When $(M_i,h_i)$
is a compact connected simply connected simple Lie group with
bi-invariant metric or a flat euclidian space, the assertion is obvious
because the group of right translations or the parallel translations
are CW-translations for $h_i$, and its Lie algebra provides Killing
fields of constant length for $h_i$ which commutes with $W_i$. If
$(M_i,h_i)$ is an odd dimensional sphere with constant curvature, the assertion
follows immediately from the fact (\cite{DM2}) that the
Randers space with navigation data $(h_i,\lambda W_i)$ is
CW-homogeneous, where $\lambda$ is any constant such that
$||\lambda W_i||_{h_i}<1$.

When $||W||_h<1$, the set of the vector fields $X-W$, where $X$ is a
Killing vector field of constant length $1$ with respect to $h$
commuting with $W$,  exhaust all tangent directions. It is easily seen that any vector field
$X-W$ as above is a Killing vector field of constant length $1$ for $F$. By
Theorem \ref{2}, $(M,F)$ is restrictively CW-homogeneous.

Finally, we prove the CW-homogeneity for a Randers space $(M,F)$ as
described in the theorem. If $(M,F)$ is not CW-homogeneous, then we
can find a Killing vector field of the form $X-W$ of constant length $1$ with
respect to $F$, where $X$ is a Killing vector field of constant
length $1$ for $h$ and $[X,W]=0$. Moreover, there is  a positive constant $t_0$ and two
pairs of points $x_i$ and $x'_i$, $i=0,1$, such that the
diffeomorphisms $\varphi_{X-W;t}$ generated by $X-W$ satisfy the
condition $\varphi_{X-W;t_0}(x_i)=x'_i$, $i=0,1$,  and the geodesic
flow curve of $X-W$ is minimizing from $x_0$ to $x'_0$ but not
minimizing from $x_1$ to $x'_1$. Since $[X,W]=0$,  the
diffeomorphisms $\varphi_{X;t}$ commute with  $\varphi_{W;t'}$. Furthermore,
 we have
\begin{equation}
\varphi_{X-W;t}=\varphi_{X;t}\varphi_{W;-t}=\varphi_{W;-t}\varphi_{X;t}.
\end{equation}
 Since the flow curve from $x_1$ to $x'_1$ is not minimizing,
there is a Killing vector field $X'-W$ of constant length $1$ for
$F$, where $X'$ is a Killing vector field of
constant length $1$ for $h$ with  $[X',W]=0$, and a constant $t'\in [0,t_0)$ such that
\begin{equation}
\varphi_{W;-t_0}\varphi_{X;t_0}(x_1)=\varphi_{W;-t'}\varphi_{X';t'}(x_1),
\end{equation}
i.e., $\varphi_{X;t_0}(x_1)=\varphi_{W;t_0-t'}\varphi_{X',t'}(x_1)$. Since
 $h$ is a symmetric Riemannian metric,  and since $X$ is a Killing field of
constant length $1$ with respect to $h$, the  centralizer of $X$ in the full group of isometries of $h$ acts transitively on
$M$ (see \cite{WO62, OZ74}). Thus there is an
isometry $g$ of $h$, which sends $x_1$ to $x_0$, such that $\mathrm{Ad}(g) X=X$.  Then we have
\begin{equation}
\varphi_{\mathrm{Ad}(g) X;t_0}(x_0)=\varphi_{X;t_0}(x_0)=\varphi_{\mathrm{Ad}(g) W;t_0-t'}\varphi_{\mathrm{Ad}(g)
X';t'}(x_0).
\end{equation}
Since the right side of the above equality gives a path with  arc
length $t'+(t_0-t')|W|_h<t_0$,  the geodesic $\varphi_{X;t}(x_0)$
from $t=0$ to $t=t_0$ is not minimizing with respect to $h$. Now the
next lemma asserts that the above equality still holds with $g$
changed to $e$, for some other $X'$ and $t'$ satisfying the same
properties as above. This implies that the geodesic
$\varphi_{X-W;t}(x_0)$ for $F$ is not minimizing between  $t=0$ to
$t=t_0$, which is a contradiction. Therefore $(M, F)$ must be
CW-homogeneous.

\begin{lemma}
Let $(M,h)$ be connected simply connected Riemannian CW-homogeneous
space, $X$ be a Killing vector field of constant length $1$ for $h$,
and $W$ be a Killing field of constant length less than $1$. Assume
the geodesic $\varphi_{X;t}(x_0)$ of $h$ from $t=0$ to $t=t_0$ is
not minimizing. Then there is a $t'\in [0,t_0)$ and a Killing vector
field $X'$ of constant length $1$ for $h$ with $[X',W]=0$, such that
\begin{equation}
\varphi_{X;t_0}(x_0)=\varphi_{W;t_0-t'}\varphi_{X';t'}(x_0).
\end{equation}
\end{lemma}

\begin{proof}
Let $f(t)$ be the distance from $x_0$ to
$\varphi_{W;t}\varphi_{X;t_0}(x_0)$ with respect to $h$. Obviously
$f(t)$ is continuous and non-negative for all $t\geq 0$ and
$f(0)<t_0$ because the geodesic $\varphi_{X;t}(x_0)$ from $t=0$ to
$t=t_0$ is not minimizing. There exists a $t'\in [0,t_0)$ such that
$f(t_0-t')=t'$. The completeness of $M$ implies there is a
minimizing geodesic from $x_0$ to
$\varphi_{W;t_0-t'}\varphi_{X;t_0}(x_0)$. Our earlier arguments indicate that
this minimizing geodesic is the flow curve of a Killing vector field
$X'$ of constant length 1 with respect to $h$ with $[X',W]=0$, i.e.
\begin{equation}
\varphi_{X;t_0}(x_0)=\varphi_{W;t_0-t'}\varphi_{X';t'}(x_0).
\end{equation}
This completes the proof of the lemma as well as of the theorem.
\end{proof}

\end{document}